\documentclass[12pt, a4paper, oneside]{amsart}
\usepackage{amsmath}
\usepackage{amsthm}
\usepackage{amssymb}
\def\overset#1#2{{\mathrel{\mathop {{#2}_{}}\limits^{#1}}}}
\def\underset#1#2{{\mathrel{\mathop {{}_{} {#2}}\limits_{{#1}_{}}}}}
\def\upplim_#1{\underset{#1}{\overline\lim}\;}
\def\lowlim_#1{\underset{#1}{\underline\lim}\;}
\newtheorem{lemma}[equation]{Lemma}

\newtheorem{proposition}[equation]{Proposition}

\newtheorem{theorem}[equation]{Theorem}

\newcommand{\C}{{\mathbb{C}}}

\newcommand{\N}{\mathbb{N}}
\newcommand{\B}{\mathbb{B}}

\renewcommand{\P}{{\mathbb{P}}}

\newcommand{\R}{{\mathbb{R}}}

\numberwithin{equation}{section}

\begin{document}
\title[Meromorphic mappings on K\"{a}hler manifolds weakly sharing hyperplanes]{Meromorphic mappings on K\"{a}hler manifolds weakly sharing hyperplanes in $\P^n(\C)$} 

\author{Si Duc Quang}  
\def\thefootnote{\empty}

\begin{abstract}
In this paper, we study the uniqueness problem for linearly nondegenerate meromorphic mappings from a K\"{a}hler manifold into $\P^n(\C)$ satisfying a condition $(C_\rho)$ and sharing hyperplanes in general position, where the condition that two meromorphic mappings $f,g$  have the same inverse image for some hyperplanes $H$ is replaced by a weaker one that $f^{-1}(H)\subset g^{-1}(H)$. Moreover, we also give some improvements on the uniqueness problem and algebraic dependence problem of meromorphic mappings which share hyperplanes and satisfy $(C_\rho)$ conditions for different non-negative numbers $\rho$. 
\end{abstract}
\maketitle

\section{Introduction}
Let $f: M \rightarrow \P^n(\C)$ be a meromorphic mapping of an $m$-dimensional complete connected K\"{a}hler manifold $M$, whose universal covering is biholomorphic to a ball $\B(R_0)=\{z \in \C^m ;\|z\|<R_0\}\ (0<R_0\le +\infty)$, into $\P^n(\C)$. For $\rho \geq 0$, we say that $f$ satisfies the condition $(C_{\rho})$ if there exists a nonzero bounded continuous real-valued function $h$ on $M$ such that
$$\rho \Omega_{f}+{\rm dd}^{c}\log h^{2} \geq \operatorname{Ric} \omega,$$
where $\Omega_{f}$ is the full-back of the Fubini-Study form $\Omega$ on $\P^n(\C), \omega=\frac{\sqrt{-1}}{2} \sum_{i, j} h_{i \bar{j}} d z_{i} \wedge$ $d \bar{z}_{j}$ is the K\"{a}hler form on $M, \operatorname{Ric} \omega=d d^{c} \log \left(\operatorname{det}\left(h_{i \bar{j}}\right)\right), d=\partial+\bar{\partial}$ and $d^{c}=\frac{\sqrt{-1}}{4 \pi}(\bar{\partial}-$ $\partial)$.

Take a local holomorphic coordinates $(U,z)$ of $M$, where $z=(z_1,\ldots,z_n)$ and $U$ is a Cousin II domain. Let $f=(f_0:\cdots:f_n)$ be a (local) reduced representation of $f$ on $U$. Denote by $\mathcal M_p$ the field of all germs of meromorphic functions at a point $p\in U$. Denote by $\mathcal F^{\kappa}$ the $\mathcal M_p$-submodule of $\mathcal M^{n+1}_p$ generated by $\{(\mathcal D^{\alpha}f_0,\ldots,\mathcal D^{\alpha}f_n):|\alpha|\le \kappa\}$ for $\kappa>0$ and $\mathcal F^{-1}=\{0\}$. Here $|\alpha|=\sum_{j=1}^m\alpha_j$ and $\mathcal D^{\alpha}\varphi=\frac{\partial^{|\alpha|}\varphi}{\partial z_1^{\alpha_1}\ldots\partial z_m^{\alpha_m}}$ for each meromorphic function $\varphi$ on $U$ and $\alpha=(\alpha_1,\ldots,\alpha_m)\in\mathbb N^m$. We define
\begin{align*}
&r_f(k)=\text{rank}_{\mathcal M_p}\mathcal F_p^k-\text{rank}_{\mathcal M_p}\mathcal F_p^{k-1}\ (k\ge 0),\\ 
&r_f:=\sum_{k\ge 0}r_f(k)-1\ \ \text{and }\  \ell_f:=\sum_{k\ge 0}kr_f(k),\\
&m_f:=\sum_{k, l}(k-l)^{+} \min \left\{{ }_{n-1} H_l,\left(r_f(k)-\sum_{\lambda=0}^{l-1}{ }_{n-1} H_\lambda\right)^{+}\right\},
\end{align*}
where $x^{+}=\max \{x, 0\}$ for a real number $x$ and ${ }_{n-1} H_\lambda$ denotes the number of repeated combinations of $\lambda$ elements among $n-1$ elements. One has
$$0 \le m_f \le \ell_f \le \frac{n(n+1)}{2}.$$

Let $H$ be a hyperplane in $\P^n(\C)$, we (throughout this paper) also denote by the same letter $H$ a linear form defining $H$, i.e., we may write
$$H(x_0,\ldots,x_n)=\sum_{j=0}^na_{1j}x_j,$$
Let $f=(f_0:\cdots:f_n)$ be a locally reduced representation of $f$. We set 
$$ H(f):=a_0f_0+\cdots+a_nf_n. $$
Then, the function $H(f)$ is locally defined and depends on the choice of the local reduced representation of $f$. However, its zero divisor $\nu_{H(f)}$ does not depend on this choice and hence it is globally well-defined.

Two meromorphic functions $f$ and $g$ from $M$ into $\P^n(\C)$ are said to share the hyperplane $H$ without multiplicity if $f^{-1}(H)=g^{-1}(H)$ and $f=g$ on $f^{-1}(H)$. In 1986, H. Fujimoto \cite{Fu86} proved the following uniqueness theorem for meromorphic mappings on a complete K\"{a}hler manifold sharing a family of hyperplanes in $\P^n(\C)$ in general position as follows.

\vskip0.1cm 
\noindent
{\bf Theorem A} (see \cite[Main Theorem]{Fu86}). {\it Let $M$ be a complete, connected K\"{a}hler manifold whose universal covering is biholomorphic to $\B(R_0) \subset \C^m\ (0<R_0\le +\infty)$, and let $f$ and $g$ be linearly nondegenerate meromorphic maps of $M$ into $\P^n(\C)$. If $f$ and $g$ satisfy the condition $(C_\rho)$ and there exist hyperplanes $\{H_i\}_{i=1}^q$ of $\P^n(\C)$ in general position such that
\begin{itemize}
\item[(i)] $f^{-1}(H_i)=g^{-1}(H_i)\ (1\le i\le q)$ and $f=g$ on $\cup_{i=1}^q f^{-1}(H_i),$
\item[(ii)] $q>n+1+\rho(\ell_f+\ell_g)+m_f+m_g,$
\end{itemize}
then $f\equiv g$.}

Hence, Theorem A implies the previous uniqueness results for meromorphic mappings from $\C^m$ into $\P^n(\C)$ given by R. Nevanlinna \cite{N}, L. Smiley \cite{S} and S. J. Drouilhet \cite{D}. Recently, K. Zhou and L. Jin \cite{Zhou} considered the case of meromorphic mappings from $\C^m$ into $\P^n(\C)$ where the condition ``$f^{-1}(H_i)=g^{-1}(H_i)$'' is replaced by a weaker one that $f^{H_i}\subset g^{-1}(H_i)$ for some hyperplane $H_i$. They proved the following.

\vskip0.1cm
\noindent
{\bf Theorem B} (see \cite[Theorem 1.1]{Zhou}). {\it Let $f, g: \C^m \rightarrow \P^n(\C)$ be meromorphic maps. Let $H_1, \ldots, H_q$ be hyperplanes of $\P^n(\C)$ in general position with $f(\C^m) \nsubseteq H_j, g(\C^m) \nsubseteq H_j$ for $1 \leq j \leq q$ and $\dim f^{-1}(H_i \cap H_j) \leq m-2$ for $1 \leq i<j \leq q$. Suppose that:
\begin{itemize}
\item[(a)] $f^{-1}(H_j)=g^{-1}(H_j)$ for $1 \leq j \leq p$, and $f^{-1}(H_j) \subseteq g^{-1}(H_j)$ for $p<j \leq q$,
\item[(b)] $f\equiv g$ on $\bigcup_{j=1}^q f^{-1}(H_j)$.
\end{itemize}
Then $f=g$ if any one of the following conditions is satisfied:
\begin{itemize}
\item[(i)] $f$ or $g$ is nonconstant and $p=2 n+2, q>3 n+3-2 \sqrt{n}$.
\item[(ii)] $f$ or $g$ is linearly nondegenerate and $p=2 n+2, q \geq 2 n+3$.
\item[(iii)] $f$ or $g$ is nonconstant and $p=2 n+1, q \geq 4 n+3$.
\item[(iv)] Both $f$ and $g$ are linearly nondegenerate and $p=n+2, q \geq n^3+n^2+n+4$.
\end{itemize}}

Motivated by the work of K. Zhou and L. Jin, our first aim in this paper is to extend the above mentioned results to the case of meromorphic mappings on K\"{a}hler manifold. Namely, we will prove the following result.
\begin{theorem}\label{1.1}
Let $M$ be a complete, connected K\"{a}hler manifold whose universal covering is biholomorphic to $\B(R_0) \subset \C^m\ (0<R_0\le +\infty)$. Let $f, g: M \rightarrow\P^n(\C)$ be linearly nondegenerate meromorphic mappings satisfying a condition $(C_{\rho})$. Let $H_1,\ldots,H_q$ be hyperplanes of $\P^n(\C)$ in general position with $\dim f^{-1}(H_i\cap H_j)\le m-2$ for every $1\le i<j\le q$, such that
\begin{itemize}
\item[(i)] $f^{-1}(H_i)=g^{-1}(H_i)\ \forall\ 1\le i\le p$, $f^{-1}(H_i)\subset g^{-1}(H_i)\ \forall\ p+1\le i\le q$, 
\item[(ii)] $f=g$ on $\bigcup_{i=1}^qf^{-1}(H_i)$, 
\end{itemize}
where $n+2\le p\le 2n+2$. Then $f\equiv g$ if 
$$q>2n+2+pn\left(\frac{n+1}{p-n-1}-1\right)+2\rho\left(\ell_f+\frac{n+1}{p-n-1}\ell_g\right)$$
$$\text{or }\ q>2n+1+pn\left(\frac{n}{p-n-1}-\frac{n-1}{n}\right)+2\rho\left(\ell_f+\frac{n}{p-n-1}\ell_g\right).$$
\end{theorem}

\noindent
\textbf{\textit{Remark 1.}} The condition of $p$ and $q$ is fullfiled in the following cases:
\begin{itemize}
\item[(1)] $p=2n+2$, $q> 2n+2+2\rho(\ell_f+\ell_g)$. 
\item[(2)] $p=2n+1$, $q> 4n+2+2\rho(\ell_f+\ell_g)$. 
\item[(3)] $p=2n$, $q>6n+2+\frac{2}{n-1}+2\rho\left(\ell_f+\frac{n}{n-1}\ell_g\right)$, for $n\ge 2$.
\item[(4)] $p=n+2$, $q>n^3+n^2+n+3+2\rho\left(\ell_f+n\ell_g\right)$. 
\end{itemize}
Then, our result implies the above mentioned result of K. Zhou and L. Jin for the case of linearly nondegenerate meromorphic mappings. In order to prove the above result, we have to develop our previous method on ``funtions of small integration'' and ``functions of bounded integeration'' in \cite{Q19,Q21}. We will prove a key lemma (see Lemma \ref{3.1} in Section 3), which generalizes and improves Proposition 3.5 of \cite{Q21}, and apply it to estimate the divisor of the auxilliary functions.  

With the useful of Lemma \ref{3.1}, we may improve the previous result on the uniqueness problem and the algebraic degeneracy problem of meromorphic mappings on K\"{a}hler manifolds. Moreover, we may consider the case of meromorphic mappings satisfying the condition $(C_\rho)$ with different numbers $\rho$.  Namely, we will prove the following result.

\begin{theorem}\label{1.2}
Let $M$ be as in Theorem \ref{1.1}. Let $f, g: M \rightarrow\P^n(\C)$ be linearly nondegenerate meromorphic mappings, which satisfy the condition $(C_{\rho_f})$ and $(C_{\rho_g})$ for non-negative constants $\rho_f$ and $\rho_g$, respectively. Let $H_1,\ldots,H_q$ be $q$ hyperplanes of $\P^n(\C)$ in general position with $\dim f^{-1}(H_i\cap H_j)\le m-2$ for every $1\le i<j\le q$, such that
\begin{itemize}
\item[(i)] $\nu^{[\ell]}_{H_i(f)}=\nu^{[\ell]}_{H_i(g)}$ for every $i=1,\ldots,q,$
\item[(ii)] $f=g$ on $\bigcup_{1\le i\le q}(f^1)^{-1}(H_i)$. 
\end{itemize}
Then $f\equiv g$ if any one of the following conditions is satisfied:
\begin{itemize}
\item[(a)] $\ell=1$ and $q>2n+2+\rho_{f,g}(\ell_f+\ell_g),$ where $\rho_{f,g}=2\frac{\rho_f\cdot\rho_g}{\rho_f+\rho_g}$.
\item[(b)] $\ell\ge n+1$ and $q>2n+1+\max\big\{\frac{3(p-2)p-(\ell-n)}{3(p-2)p+(\ell-n)/n},\frac{2n-1}{2n-1/n}\big\}$ $+\rho_{f,g}(\frac{4n^2-1}{2n^2-1}(\ell_f+\ell_g)+\ell-n),$ where $p=\binom{n+1}{2n+2}$. 
\end{itemize} 
\end{theorem}

For the case of more than two meromorphic mappings sharing a family of hyperplanes, we prove the following two results.
\begin{theorem}\label{1.3}
Let $M$ be as in Theorem \ref{1.1}. Let $f^1,f^2,f^3: M \rightarrow\P^n(\C)\ (n\ge 2)$ be linearly nondegenerate meromorphic mappings, which satisfy the condition $(C_{\rho_{f^1}}),(C_{\rho_{f^2}})$ and $(C_{\rho_{f^3}})$ for non-negative constants $\rho_{f^1},\rho_{f^2}$ and $\rho_{f^3}$, respectively. Let $H_1,\ldots,H_q$ be $q$ hyperplanes of $\P^n(\C)$ in general position with $\dim (f^1)^{-1}(H_i\cap H_j)\le m-2$ for every $1\le i<j\le q$. Assume that $f=g$ on $\bigcup_{\overset{1\le u\le 3}{1\le i\le q}}(f^u)^{-1}(H_i)$. Then $f^1\wedge f^2\wedge f^3=0$ if $ q>n+1+\frac{3nq}{2q+2n-2}+2\left(\sum_{u=1}^3\frac{1}{\rho_{f^u}}\right)^{-1}\sum_{u=1}^3\ell_{f^u}.$
\end{theorem}

\begin{theorem}\label{1.4}
Let $M$ be as in Theorem \ref{1.1}. Let $f^1,\ldots,f^k: M \rightarrow\P^n(\C)$ be $k$ linearly nondegenerate meromorphic mappings, which satisfy the conditions $(C_{\rho_{f^u}}),\ldots,(C_{\rho_{f^k}})$ for non-negative constants $\rho_{f^1},\ldots,\rho_{f^k}$, respectively. Let $H_1,\ldots,H_q$ be $q$ hyperplanes of $\P^n(\C)$ in general position with $\dim (f^1)^{-1}(H_i\cap H_j)\le m-2$ for every $1\le i<j\le q$ such that
\begin{itemize}
\item[(i)] $\nu^{[n]}_{H_i(f^1)}=\cdots=\nu^{[n]}_{H_i(f^u)}$ for every $i=1,\ldots,q,$
\item[(ii)] $f=g$ on $\bigcup_{1\le i\le q}(f^1)^{-1}(H_i)$. 
\end{itemize}
Then $f^1\wedge \cdots\wedge f^k=0$ if $q>n+1+\frac{knq}{(k-1)q+k(n-1)}+2\left(\sum_{u=1}^k\frac{1}{\rho_{f^u}}\right)^{-1}\sum_{u=1}^k\ell_{f^u}.$
\end{theorem}

\textbf{\textit{Remark 2.}} (a) Theorems \ref{1.2}, \ref{1.3} and \ref{1.4} generalize and improve the recent results in \cite{Q21} (Lemma 4.3, Theorem 1.2) and \cite{Q22} (Theorem 1.2).

(b) In this paper, we only concentrate on linearly nondegenerate meromorphic mappings. But our method can be applied to study the case of nonconstant meromorphic mappings. However, in that case the computation may be much more complicate, since there are many more parameters appearing. Then, that problem is still an interesting open question.
 
\section{Auxiliary results}
Let $E$ be a divisor on $\B(R_0)$, which is usually regarded as a function from $\B(R_0)$ into $\mathbb Z$. The support $\operatorname{Supp}(E)$ is defined as the closure of the set $\{z|E(z)\ne 0\}$. 
For a positive integer $k$ (may be $+\infty$), we define $E^{[k]}(z)=\min\{E(z),k\}$ and
$$n^{[k]}(t, E):= 
\begin{cases}\int_{\operatorname{Supp}(E) \cap B(t)} E^{[k]} v_{m-1} & \text { if } m \geq 2, \text{ where }v_{m-1}(z)=\left(dd^c\|z\|^2\right)^{m-1},\\
 \sum_{|z| \leq t} E^{[k]}(z) & \text { if } m=1.
\end{cases}$$
The truncated counting function to level $k$ of $E$ is defined by
$$N^{[k]}\left(r, r_{0} ; E\right):=\int_{r_{0}}^{r} \frac{n^{[k]}(t, E)}{t^{2 m-1}} d t \quad\left(r_{0}<r<R_{0}\right).$$
We omit the character ${ }^{[k]}$ if $k=+\infty$.

Let $\varphi$ be a non-zero meromorphic function on $\B(R_0)$. We denote by $\nu_{\varphi}^{0}$ (resp. $\nu_{\varphi}^{\infty}$ ) the divisor of zeros (resp. divisor of poles ) of $\varphi$ and set $\nu_{\varphi}=\nu_{\varphi}^{0}-\nu_{\varphi}^{\infty}$.
For convenience, we will write $N_{\varphi}\left(r, r_{0}\right)$ and $N_{\varphi}^{[k]}\left(r, r_{0}\right)$ for $N\left(r, r_{0} ; \nu_{\varphi}^{0}\right)$ and $N^{[k]}\left(r, r_{0} ; \nu_{\varphi}^{0}\right)$, respectively.

Let $f:\B(R_0) \longrightarrow \P^n(\C)$ be a meromorphic mapping. Fix a homogeneous coordinates system $\left(w_{0}: \cdots: w_{n}\right)$ on $\P^n(\C)$. We take a reduced representation $f=\left(f_{0}: \cdots: f_{n}\right)$ and set $\|f\|=\left(\left|f_{0}\right|^{2}+\cdots+\left|f_{n}\right|^{2}\right)^{1 / 2}$. 
The characteristic function of $f$ is defined by
$$T_f(r, r_{0})=\int_{\|z\|=r} \log \|f\| \sigma_{m}-\int_{\|z\|=r_0} \log \|f\| \sigma_{m}, \quad 0<r_{0}<r<R_{0},$$
where $\sigma_{m}(z)=d^c \log \|z\|^2\wedge(dd^c\log \|z\|^{2})^{m-1}.$ Here and throughout this paper, we assume that the numbers $r_{0}$ and $R_{0}$ are fixed with $0<r_{0}<R_{0}$. 

Suppose that $f$ is linearly nondegenerate. Then, there is an $n+1$-tuple $\alpha=(\alpha_0,\ldots,\alpha_n)\in (\mathbb N^m)^{n+1}$ such that $(\mathcal D^{\alpha_0}(f),\ldots,\mathcal D^{\alpha_{\ell (k)-1}})$ is a basis of $\mathcal M_p$-module $\mathcal F^{\ell (k)}$, where $\ell (k)=\dim_{\mathcal M_p}\mathcal F^{k}$ for all $k=1,\ldots,k_0=\min\{k:\ell (k)=n+1\}$. The tuple $\alpha=(\alpha_0,\ldots,\alpha_n)$ is called the admissible set of $f$ and 
$$W_{\alpha_0,\ldots ,\alpha_n}(f_0,\ldots ,f_n):=\det (\mathcal D^{\alpha_j}(f_i))_{0\le j,i\le n}$$ 
is called the generalized Wronskian of $f$. We note that $|\alpha|=\sum_{i=0}^n|\alpha_i|=\ell_f.$
\begin{proposition}[{see \cite[Proposition 2.12]{Fu86}}]\label{2.1}
 Let $H_1,\ldots ,H_q$ be $q$ hyperplanes in $\P^n(\C)$ in general position. Let $f$ be a linearly nondegenerate meromorphic mapping from the ball $\B^m(R_0)\subset\C^m$ into $\P^{n}(\C)$ with a reduced representation $f=(f_0,\ldots ,f_{n})$ and let $(\alpha_0,\ldots ,\alpha_n)$ be an admissible set of $f$. Then, for $0 < r_0 < R_0$ and $0< t\ell_f< p<1$,  there exists a positive constant $K$ such that for $r_0 < r < R < R_0$,
$$\int_{S(r)}\biggl |z^{\alpha_0+\cdots +\alpha_n}\dfrac{W_{\alpha_0,\ldots ,\alpha_n}(f_0,\ldots ,f_n)}{H_1(f)\ldots H_{q}(f)}\biggl |^t\cdot \|f\|^{t(q-n-1)}\sigma_m\le K\biggl (\dfrac{R^{2m-1}}{R-r}T_f(R,r_0)\biggl )^p.$$
\end{proposition}

\section{Main lemmas}
Let $f^1,\ldots,f^k$ be $k$ meromorphic mappings from $\B^m(R_0)$ into $\P^n(\C)$.
We fix a reduced representation $f^u=(f^u_0:\cdots :f^u_n)$ of $f^u$ and set $\|f^u\|=(|f^u_0|^2+\cdots+|f^u_n|^2)^{1/2}$ for $u=1,\ldots,k$.

Denote by $\mathcal C(\B^m(R_0))$ the set of all functions $g: \B^m(R_0)\to [0,+\infty]$ which is continuous outside an analytic set of codimension two and only attain $+\infty$ in an analytic thin set.

For a non negative integer $l_0$, we denote by $S(l_0;f^1,\ldots,f^k)$ the set of all functions $g$ in $\mathcal C(\B^m(R_0))$ such that there exist an element $\alpha=(\alpha_1,\ldots,\alpha_m)\in\N^m$ with $|\alpha|\le l_0$ satisfying: for every $0\le tl_0<p<1$, there exists a positive number $K$ with
$$\int_{S(r)}|z^\alpha g|^t\sigma_m \le K\left(\frac{R^{2m-1}}{R-r}\sum_{u=1}^kT_{f^u}(r,r_0)\right)^p$$
for all $r$ with $0<r_0<r<R<R_0$, where $z^\alpha=z_1^{\alpha_1}\cdots z_m^{\alpha_m}$. 

Let $p$ be a non negative integer. Denote by $B(l_0;(f^1,p_1),\ldots,(f^k,p_k))$ the set of all meromorphic functions $h$ on $\B^m(R_0)$ such that there exists $g\in S(l_0;f^1,\ldots,f^k)$ satisfying
$$|h|\le \|f^1\|^{p_1}\cdots \|f^k\|^{p_k}\cdot g,$$
outside a proper analytic subset of $\B^m(R_0)$. We will write $B(p,l_0;f^1,\ldots,f^k)$ as in \cite{Q21} for $B(l_0;(f^1,p),\ldots,(f^k,p))$.  

We easily have the following fundamental properties of the families $S(l_0;\{f^u\}_{u=1}^k)$ and $B(l_0;\{(f^u,p_u)\}_{u=1}^k)$.
\begin{itemize}
\item If $g$ is a constant function then $g\in S(0;\{f^u\}_{u=1}^k)$.
\item If $g_i\in S(l_i;\{f^u\}_{u=1}^k)\ (1\le i\le s)$ then $\prod_{i=1}^sg_i\in S(\sum_{i=1}^sl_i;\{f^u\}_{u=1}^k)$ (see Proposition 2.1 in \cite{Q21}).
\item A meromorphic mapping $h$ belongs to $B(0,l_0;\{f^u\}_{u=1}^k)$ if and only if $|h|\in S(l_0;\{f^u\}_{u=1}^k)$.
\item If $h_i\in B(l_i;\{f^1,p^i_1\},\ldots,\{f^k,p^i_k\})\ (1\le i\le s)$ then 
$$h_1\cdots h_m\in B(\sum_{i=1}^sl_i;\{f^1,\sum_{i=1}^sp^i_1\},\ldots,\{f^k,\sum_{i=1}^sp^i_k\}).$$
\item Proposition 2.1 implies that if $W(f)$ is the generalized wronskian of a linearly nondegenerate meromorphic mapping $f:\B^m(R_0)\to\P^n(\C)$  and $H_1,\ldots,H_q$ be $q$ hyperplanes in $\P^n(\C)$ in general position then the function $\left |\dfrac{W(f)\cdot\|f\|^{q-n-1}}{H_1(f)\ldots H_{q}(f)}\right |$ belongs to $S(\ell_f;f)$.
\end{itemize}

Firstly, we prove the following key lemma.
\begin{lemma}\label{3.1} 
	Let $M=\B^m(R_0)\ (0<R_0\le +\infty)$ be a complete connected K\"{a}hler manifold. Let $k$ be a positive integer and for each $u\in\{1,\ldots,k\}$, let $f^u$ be a linearly nondegenerate meromorphic mapping from $M$ into $\P^n(\C)$, which satisfies the condition $(C_{\rho_u})$ and has a reduced representation $f^u=(f_{0}^u:\cdots :f_{n}^u)$. Let $\{H^u_1,\ldots,H^u_{q_u}\}\ (1\le u\le k)$ be $k$ families of hyperplanes of $\P^n(\C)$ in general position, where $q_1,\ldots,q_k$ are positive integers. Assume that there exists a non zero holomorphic function $h$ on $\B(\R_0)$ such that:
\begin{itemize}
\item[(a)] $ h\in B(l_0;\{(f^u,p_u)\}_{u=1}^k),$ where $l_0$ is a non-negative integer, $p_1,\ldots,p_k$ are positive constants;
\item[(b)] $\nu_h\ge\sum_{u=1}^k\lambda_u\sum_{i=1}^{q_u}\nu^{[n]}_{H^u_i(f^u)},$ where $\lambda_u (1\le u\le k)$ are positive constants. 
\end{itemize}
Then there is an index $u$ such that $\lambda_u(q_u-n-1)-p_u\le 0,$ or 
$$ \sum_{u=1}^k\left [\frac{\lambda_u(q_u-n-1)-p_u}{\rho_u}-2\lambda_u \ell_{f^u}\right ]\le 2l_0.$$
\end{lemma}
Here, we regard that $\frac{x}{0}=+\infty$ and $\frac{-x}{0}=-\infty$ for every $x>0$. 
\begin{proof} Suppose contrarily that $\lambda_u(q_u-n-1)-p_u> 0,$ for all $u=1,\ldots,k$ and
$$ \sum_{u=1}^k\left [\frac{\lambda_u(q_u-n-1)-p_u}{\rho_u}-2\lambda_u \ell_{f^u}\right ]> 2l_0.$$
We consider the following two cases.

\noindent
\textbf{\textit{Case 1:}} $R_0=+\infty$. By the second main theorem in Nevanlinna theory we have
\begin{align*}
\sum_{u=1}^k\lambda_u(q_u-n-1)T_{f^u}(r,1)&\le\sum_{u=1}^k\lambda_u\sum_{i=1}^{q_u}N^{[n]}_{H^u_i(f^u)}(r,1)+o(\sum_{u=1}^kT_{f^u}(r,1))\\
&\le N_h(r,1)+o(\sum_{u=1}^kT_{f^u}(r,1))\\
&=\sum_{u=1}^kp_uT_{f^u}(r,1)+o(\sum_{u=1}^kT_{f^u}(r,1)),
\end{align*}
for all $r\in(1,+\infty)$ outside a Lebesgue set of finite measure.
This is a contradiction.

\noindent
\textbf{\textit{Case 2:}}  $R_0<+\infty$. For each $u\ (1\le u\le k)$, since $f^u$ is linearly nondegenerate, there exists an admissible set $(\alpha_0^u,\ldots,\alpha_n^u)\in\mathbb (\N^m)^{n+1}$ such that the generalized Wronskian
	$$ W(f^u):=\det\left (\mathcal D^{\alpha_i^u}(f_j^u); 0\le i,j\le n\right)\not\equiv 0. $$
	By usual argument in Nevanlinna theory, we have
	$$\nu_h\ge \sum_{u=1}^k\lambda_u\sum_{i=1}^q\nu^{[n]}_{H^u_i(f^u)}\ge\sum_{u=1}^k\lambda_u\left (\sum_{i=1}^q\nu_{H^u_i(f^u)}-\nu_{W(f^u)}\right).$$
	Put $w_u(z):=z^{\alpha_0^u+\cdots +\alpha_n^u}\dfrac{W(f^u)}{\prod_{i=1}^{q}H^u_i(f^u)}\ (1\le u\le k)$. 

Since $h\in B(l_0;\{(f^u,p_u)\}_{u=1}^k)$, there exists a non-negative plurisubharmonic function $g\in S(l_0;f^1,\ldots,f^k)$ and $\beta =(\beta_1,\ldots,\beta_m)\in\mathbb Z^{m}_+$ with $|\beta|\le l_0$ such that 
	\begin{align}\label{3.6}
	\int_{S(r)}\left |z^\beta g\right|^{t'}\sigma_m=O\left (\frac{R^{2m-1}}{R-r}\sum_{u=1}^kT_{f^u}(r,r_0)\right )^l,
	\end{align}
for every $0\le l_0t'<l<1$ and $|h|\le \prod_{u=1}^k\|f^u\|^{p_k}g.$
	Put $ t=\frac{2}{\sum_{u=1}^k(\lambda_u(q_u-n-1)-p_u)/\rho_u}>0$ and $ \phi:=|w_1|^{\lambda_1}\cdots |w_k|^{\lambda_k}\cdot|z^\beta h|,$
	Then $a=t\log\phi$ is a plurisubharmonic function on $\B^m(R_0)$ and
	$$(\sum_{u=1}^{k}\lambda_u\ell_{f^u}+l_0)t< 1.$$
	Therefore, we may choose a positive number $p$ such that $0<\sum_{u=1}^{k}\lambda_ut<p<1.$
	
	Since $f^u$ satisfies the condition $(C_{\rho_u})$, there is a continuous plurisubharmonic function $\varphi_u$ on $\B^m(R_0)$ such that
	$$ e^{\varphi_u}dV \le \|f^u\|^{2\rho_u}v_m.$$
	Note that in this case, $\rho_u>0$ for all $u$. Then the function $\varphi=\lambda'_1\varphi_1+\cdots+\lambda'_k\varphi_k+a$ is a plurisubharmonic on $\B^m(R_0)$, where $\lambda_u'=\frac{(\lambda_u(q-n-1)-p_u)t}{2\rho_u}$. One has $\sum_{u=1}^k\lambda_u'=1$ and then
	\begin{align*}
	&e^\varphi dV=e^{\lambda'_1\varphi_1+\cdots+\lambda'_k\varphi_k+t\log\phi}dV\le C'\cdot e^{t\log\phi}\cdot\prod_{u=1}^k\|f^u\|^{2\lambda'_u\rho_u}v_m=C'\cdot|\phi|^{t}\cdot\prod_{u=1}^k\|f^u\|^{2\lambda'_u\rho_u}v_m\\
	&=C''\cdot|z^\beta g|^t\cdot\prod_{u=1}^k(|w_u|^{\lambda_ut}\|f^u\|^{2\lambda_u'\rho_u+p_ut})v_m=C''\cdot|z^\beta g|^t\cdot\prod_{u=1}^k(|w_u|\cdot\|f^u\|^{(q_u-n-1)})^{t\lambda_u}v_m,
	\end{align*}
where $C',C''$ are positive constants. Setting $x_u=\frac{\sum_{i=1}^k\lambda_i\ell_{f^i}+l_0}{\lambda_u\ell_{f^u}}$, $y=\frac{\sum_{i=1}^k\lambda_i\ell_{f^i}+l_0}{l_0}$, we have $\sum_{u=1}^k\frac{1}{x_u}+\frac{1}{l_0}=1$. By integrating both sides of the above inequality over $\B^m(R_0)$ and applying H\"{o}lder inequality, we get
	\begin{align}\nonumber
	\int_{\B^m(R_0)}e^\varphi dV&\le C''\left(\int_{\B^m(R_0)}|z^\beta g|^{ty}\right)^{1/y}\cdot\prod_{u=1}^k\left (\int_{\B^m(R_0)}(|w_u|\cdot\|f^u\|^{(q_u-n-1)})^{\lambda_utx_u}v_m\right)^{1/x_u}\\
	\label{3.8}
	\begin{split}
	&=C''\left (2m\int_0^1r^{2m-1}\int_{S(r)}|z^\beta g|^{ty}v_m\right)^{1/y}\\
	&\times\prod_{u=1}^k\left (2m\int_0^1r^{2m-1}\int_{S(r)}(|w_u|\cdot\|f^u\|^{(q_u-n-1)})^{\lambda_ux_ut}v_m\right)^{1/x_u}.
	\end{split}
	\end{align}
	
	\textit{Subcase 2.a}: We suppose that
	$$ \lim\limits_{r\rightarrow R_0}\sup\dfrac{\sum_{u=1}^kT_{f^u}(r,r_0)}{\log 1/(R_0-r)}<\infty .$$
	We see that $\lambda_utx_u\ell_{f^u}=tyl_0=t(\sum_{i=1}^k\lambda_i\ell_{f^i}+l_0)<p$. By Proposition \ref{2.1}, there exists a positive constant $K$ such that, for every $0<r_0<r<R< R_0,$ we have
	\begin{align*}
	&\int_{S(r)}\bigl (|w_u|\cdot\|f^u\|^{(q_u-n-1)}\bigl )^{\lambda_ux_ut}\sigma_m\le K\left (\dfrac{R^{2m-1}}{R-r}T_{f^u}(R,r_0)\right )^{p}\ (1\le u\le k),\\
\text{ and }&\int_{S(r)}|z^\beta g|^{ty}v_m\le K\left (\dfrac{R^{2m-1}}{R-r}T_{f^u}(R,r_0)\right )^{p}.
	\end{align*}
	Choosing $R=r+\dfrac{R_0-r}{e\max_{1\le u\le k}T_{f^u}(r,r_0)}$, we have
	$ T_{f^u}(R,r_0)\le 2T_{f^u}(r,r_0)$,
	for all $r$ outside a subset $E$ of $(0,R_0]$ with $\int_E\frac{1}{R_0-r}dr<+\infty$.
	Then, the above inequality implies that
	\begin{align*}
	&\int_{S(r)}\bigl (|w_u|\cdot\|f^u\|^{(q_u-n-1)}\bigl )^{\lambda_ux_ut}\sigma_m\le \dfrac{K'}{(R_0-r)^{p}}\left (\log\frac{1}{R_0-r}\right )^{2p}\ (1\le u\le k),\\
	\text{ and }&\int_{S(r)}|z^\beta g|^{ty}v_m\le\dfrac{K'}{(R_0-r)^{p}}\left (\log\frac{1}{R_0-r}\right )^{2p}
	\end{align*}
	for all $r$ outside $E$, and for some positive constant $K'$. The inequality (\ref{3.8}) yields that
	\begin{align*}
	\int_{\B^m(R_0)}e^udV&\le C'' 2m\int_0^{R_0}r^{2m-1}\frac{K'}{R_0-r}\left (\log\frac{1}{R_0-r}\right)^{2p}dr< +\infty.
	\end{align*}
	This contradicts the results of S.T. Yau \cite{Y76} and L. Karp \cite{K82}. 
	
	\textit{Subcase 2.b:} We suppose that 
	$$ \lim\limits_{r\rightarrow R_0} \sup\dfrac{\sum_{u=1}^kT_{f^u}(r,r_0)}{\log 1/(R_0-r)}= \infty.$$
	By \cite[Proposition 6.2]{Fu85}, we have
	\begin{align*}
	\sum_{u=1}^kp_uT_{f^u}(r,r_0)&\ge N_h(r,r_0)+S(r)\ge\sum_{u=1}^k\lambda_p\sum_{i=1}^qN_{H^u_i(f^u)}^{[n]}(r,r_0)+S(r)\\ 
	&\ge \sum_{u=1}^k\lambda_u(q_u-n-1)T_{f^u}(r,r_0)+O\bigl(\log^+\frac{1}{R_0-r}+\log^+\sum_{u=1}^kT_{f^u}(r_0,r)\bigl), 
	\end{align*}
	for every $r$ excluding a set $E$ with $\int_E\frac{dr}{R_0-r}<+\infty$. This is a contradiction.

	Hence, the supposition is false. The lemma is proved.
\end{proof}
Secondly, we prove the following generalization theorem for uniqueness problem of linearly nondegenerate meromorphic mappings on K\"{a}hler manifolds.
\begin{lemma}\label{3.4}
Let $M$ be a complete, connected K\"{a}hler manifold whose universal covering is biholomorphic to $\B(R_0) \subset \C^m$, where $0<R_{0} \leq \infty$. Let $f, g: M \rightarrow\P^n(\C)$ be linearly nondegenerate meromorphic mappings, which satisfy the condition $(C_{\rho_f}), (C_{\rho_g})$ for non-negative constants $\rho_f,\rho_g$, respectively. Let $H_1,\ldots,H_q$ be $q$ hyperplane of $\P^n(\C)$ in general position and let $n+2\le p\le 2n+2<q$. Assume that:
\begin{itemize}
\item[(a)] $\dim f^{-1}(H_i\cap H_j)\le m-2 \ \forall\ 1\leq i< j\leq  q,$
\item[(b)] $f^{-1}(H_i)=g^{-1}(H_i)\ \forall\ 1\le i\le p,\ f^{-1}(H_i)\subset g^{-1}(H_i)\ \forall\ p+1\le i\le q,$
\item[(c)] $f=g$ on $\bigcup_{i=1}^qf^{-1}(H_i).$
\end{itemize}
Then $f\equiv g$ if there exist non-negative numbers $t\le\frac{2}{n}$ and $\alpha$ such that:
\begin{itemize}
\item[(1)] $(2+t)(q-n-1)-(2n+2+p(t+\alpha))> 0,$
\item[(2)] $\left(2+\frac{\alpha}{n}\right)(p-n-1)-(2n+2)> 0,$
\item[(3)] $\frac{(2+t)(q-n-1)-(2n+2+p(t+\alpha))}{\rho_f}+\frac{\left(2+\frac{\alpha}{n}\right)(p-n-1)-(2n+2)}{\rho_g}> 2\left((2+t)\ell_f+(2+\frac{\alpha}{n})\ell_g\right),$
\end{itemize}
where $t=0$ if $p> q-n-1$ and $t=\frac{2}{n}$ if $p\le q-n-1$.
\end{lemma}

\begin{proof}
Since the universal covering of $M$ is biholomorphic to $\B(R_0), 0<R_{0} \leq+\infty$, by using the universal covering if necessary, without loss of generality we assume that $M=B(R_0) \subset \C^m$. Also, we may assume that $t$ and $\alpha$ are rational numbers.

Suppose contrarily that $f\not\equiv g$. Consider the simple graph $\mathcal H$, where the set of vertice is $\{1,2,\ldots,2n+2\}$ and the set of edges is consist of all pairs $\{i,j\}$ such that $\frac{H_i(f)}{H_j(f)}\not\equiv\frac{H_i(g)}{H_j(g)}$. Since $f\not\equiv g$, the degree of $\mathcal H$ at every vertex is at least $(2n+2)-n>\frac{2n+2}{2}$. By Dirac's theorem, $\mathcal H$ has a Hamiltonian cycle, for instance it is $(1,2,3,\ldots,2n+2,1)$. Therefore, 
$$P_{i}=H_{i}(f) H_{\sigma(i)}(g)-H_{i}(g) H_{\sigma(i)}(f)\not\equiv 0,$$
where $\sigma(i)=i+1$ for $i<q$ and $\sigma(2n+2)=1$. We easily see that 
$$\nu_{P_i}(z)\geq \sum_{j=i,\sigma(i)}\min \left\{\nu_{H_j(f)}(z), \nu_{H_j(g)}(z)\right\}+\sum_{\underset{j\ne i,\sigma(i)}{j=1}}^{q}\min \left\{\nu_{H_j(f)}(z),1\right\}$$
for all $z$ ouside the analytic subset $\bigcup_{1\le u<v\le q}f^{-1}(H_u\cap H_v)$, which is of codimension two.

Then, by setting $P=\prod_{i=1}^{2n+2} P_{i} \not \equiv 0$, we have
\begin{align*}
\nu_{P}&\geq 2\sum_{j=1}^{2n+2}\min \left\{\nu_{H_j(f)}, \nu_{H_j(g)}\right\}+2n\sum_{j=1}^{2n+2}\nu^{[1]}_{H_j(f)}+(2n+2)\sum_{j=2n+3}^q\nu^{[1]}_{H_j(f)}\\
&\geq 2\sum_{j=1}^p\left (\nu^{[n]}_{H_j(f)}+\nu^{[n]}_{H_j(g)}-n\nu^{[1]}_{H_j(f)}\right)+2n\sum_{j=1}^p\nu^{[1]}_{H_j(f)}+(2n+2)\sum_{j=p+1}^q\nu^{[1]}_{H_j(f)}\\
&= 2\sum_{j=1}^p\nu^{[n]}_{H_j(f)}+2\sum_{j=1}^p\nu^{[n]}_{H_j(g)}+(2n+2)\sum_{j=p+1}^q\nu^{[1]}_{H_j(f)}\\
&\ge  2\sum_{j=1}^q\nu^{[n]}_{H_j(f)}+2\sum_{j=1}^p\nu^{[n]}_{H_j(g)}+2\sum_{j=p+1}^q\nu^{[1]}_{H_j(f)}\\
&\ge  (2+t)\sum_{j=1}^q\nu^{[n]}_{H_j(f)}-(t+\alpha)\sum_{j=1}^p\nu^{[n]}_{H_j(f)}+\alpha\sum_{j=1}^p\nu^{[1]}_{H_j(f)}+2\sum_{j=1}^p\nu^{[n]}_{H_j(g)}\\
&\ge (2+t)\sum_{j=1}^q\nu^{[n]}_{H_j(f)}-(t+\alpha)\sum_{j=1}^p\nu_{H_j(f)}+\left(2+\frac{\alpha}{n}\right)\sum_{j=1}^p\nu^{[n]}_{H_j(g)},
\end{align*}

Take a positive integer $k$ so that $k(t+\alpha)$ is an integer and consider the holomorphic function $\tilde P=P^k\cdot\prod_{j=1}^{p}H_j(f)^{k(t+\alpha)}.$ It is clear that
$$ \nu_{\tilde P}\ge k(t+\alpha)\sum_{j=1}^p\nu_{H_j(f)}+k\nu_P\ge k(2+t)\sum_{j=1}^q\nu^{[n]}_{H_j(f)}+k\left(2+\frac{\alpha}{n}\right)\sum_{j=1}^p\nu^{[n]}_{H_j(g)}$$
and $\tilde P\in B\big(0;(f,(2n+2+p(t+\alpha))k),(g,(2n+2)k)\big)$. By Lemma \ref{3.1}, one of the following must hold:
\begin{itemize}
\item $(2+t)(q-n-1)-(2n+2+p(t+\alpha))\le 0,$
\item $\left(2+\frac{\alpha}{n}\right)(p-n-1)-(2n+2)\le 0,$
\item $\frac{(2+t)(q-n-1)-(2n+2+p(t+\alpha))}{\rho_f}+\frac{\left(2+\frac{\alpha}{n}\right)(p-n-1)-(2n+2)}{\rho_g}\le 2\left((2+t)l_f+(2+\frac{\alpha}{n})l_g\right)$.
\end{itemize}
This is a contradiction. Therefore, we must have $f\equiv g$. The lemma is proved.
\end{proof}

\section{Uniqueness problem}
\begin{proof}[Proof of Theorem \ref{1.1}] By Theorem \ref{3.4}, in order to show that $f\equiv g$ we just show the existence of the non-negative numbers $t\le\frac{2}{n}$ and $\alpha$ satisfying the inequalities (1), (2), (3) of Theorem \ref{3.4}. We choose
$$ \alpha:=n\left(\frac{2n+2+\epsilon}{p-n-1}-2\right),$$
where $\epsilon$ is a positive constant small enough. Then the inequality (2) of Theorem \ref{3.4} is satisfied. The inequalities (1) and (3) of Theorem \ref{3.4} become
\begin{align}\label{4.1}
(2+t)(q-n-1)-2n-2-p\left(t-2n+n\frac{2n+2+\epsilon}{p-n-1}\right)> 0,
\end{align}
and
\begin{align}
\label{4.2}
\begin{split}
(2+t)(q-n-1)-2n-2-p\left(t-2n+n\frac{2n+2+\epsilon}{p-n-1}\right)+\epsilon\\
>2\rho\left((2+t)\ell_f+\frac{2n+2+\epsilon}{p-n-1}\ell_g\right).
\end{split}
\end{align}
Therefore, in order to show the existence of such $t$ and $\alpha$ (equivalent to the existence of $\epsilon >0$) we just to show that there is $t\in[0,\frac{2}{n}]$ such that:
\begin{align}
\label{4.3}
\begin{split}
(2+t)(q-n-1)-2n-2-&p\left(t-2n+n\frac{2n+2}{p-n-1}\right)\\
&>2\rho\left((2+t)\ell_f+\frac{2n+2}{p-n-1}\ell_g\right).
\end{split}
\end{align}
If put $t=0$ then the inequality (\ref{4.3}) is equivalent to that 
\begin{align*}
q-n-1&>n+1+pn\left(\frac{n+1}{p-n-1}-1\right)+2\rho\left(\ell_f+\frac{n+1}{p-n-1}\ell_g\right)\\ 
\Leftrightarrow\ q&>2n+2+pn\left(\frac{n+1}{p-n-1}-1\right)+2\rho\left(\ell_f+\frac{n+1}{p-n-1}\ell_g\right).
\end{align*}
If put $t=\frac{2}{n}$ then the inequality (\ref{4.3}) is equivalent to that 
\begin{align*}
\frac{n+1}{n}(q-n-1)&>n+1+p\left(\frac{1-n^2}{n}+n\frac{n+1}{p-n-1}\right)+2\rho\left(\frac{n+1}{n}\ell_f+\frac{n+1}{p-n-1}\ell_g\right)\\
\Leftrightarrow\ q&>2n+1+pn\left(\frac{n}{p-n-1}-\frac{n-1}{n}\right)+2\rho\left(\ell_f+\frac{n}{p-n-1}\ell_g\right).
\end{align*}
Then, with the asumption of the themrem, the inequality (\ref{4.3}) is satisfied for $t=0$ or $t=\frac{2}{n}$. Hence, $f\equiv g$.
The proof of the theorem is completed.
\end{proof}

In order to prove Theorem \ref{1.2}, we need the following proposition of H. Fujimoto \cite{Fu75}. 

\begin{proposition}[{See H. Fujimoto \cite{Fu75}}]\label{new2} 
Let $G$ be a  torsion free abelian group and $A=(a_1,\ldots,a_q)$ be a $q-$tuple of elements $a_i$ in $G$. If $A$ has the property $(P_{r,s})$ for some $r,s$ with $q\ge r>s>1,$ then there exist $i_1,\ldots,i_{q-r+2}$ with $1\le i_1<\cdots<i_{q-r+2}\le q$ such that 
$a_{i_1}=a_{i_2}=\cdots=a_{i_{q-r+2}}.$
\end{proposition}
Here, the $q$-tuple $A$ has the property $(P_{r,s})$ if any $r$ elements $a_{l(1)},\ldots,a_{l(r)}$ in $A$ satisfy the condition that for any given $i_1,\ldots,i_s \ (1\le i_1<\cdots<i_s\le r)$, there exist $j_1,\ldots,j_s \ (1\le j_1<\cdots<j_s\le r)$ with $\{i_1,\ldots,i_s\}\neq\{j_1,\ldots,j_s\}$ such that $a_{l(i_1)}\cdots a_{l(i_s)}=a_{l(j_1)}\cdots a_{l(j_s)}.$

\begin{proof}[Proof of Theorem \ref{1.2}]
Similarly as in the proof of Lemma \ref{3.4}, we may suppose that $M=\B(R_0)\subset \C^m$. Suppose contrarily that $f\not\equiv g$. 

Consider the simple graph $\mathcal H$, where the set of vertice is $\{1,2,\ldots,q\}$ and the set of edges consists of all pair $\{(i,j)|\frac{H_i(f)}{H_j(f)}\not\equiv\frac{H_i(g)}{H_j(g)}\}$. Since $f\not\equiv g$, the degree of $\mathcal H$ at every vertex is at least $q-n>\frac{q}{2}$. By Dirac's theorem, $\mathcal H$ has a Hamiltonian cycle, for instance it is $(1,2,3,\ldots,q,1)$. Therefore 
$$P_{i}=H_{i}(f) H_{\sigma(i)}(g)-H_{i}(g) H_{\sigma(i)}(f)\not\equiv 0,$$
where $\sigma(i)=i+1$ for $i<q$ and $\sigma(q)=1$. We easily see that 
$$\nu_{P_i}(z)\geq \sum_{j=i,\sigma(i)}\min \left\{\nu_{H_j(f)}(z), \nu_{H_j(g)}(z)\right\}+\sum_{\underset{i\ne j\ne \sigma(i)}{j=1}}^{q}\min \left\{\nu_{H_j(f)}(z),1\right\}$$
for all $z$ ouside the analytic subset $\bigcup_{1\le u<v\le q}f^{-1}(H_u\cap H_v)$, which is of codimension two.

Define $\nu_i=\min\{1,|\nu_{H_i(f)}-\nu_{H_i(g)}|\}$ and $\ell'=\max\{0,\ell-n\}$. We see that if $\nu_i(z)\ne 0$ then $\min\{\nu_{H_i(f)},\nu_{H_i(g)}\}\ge \ell$. Then, by setting $P=\prod_{i=1}^{q} P_{i} \not \equiv 0$, we have
\begin{align}\label{new}
\begin{split}
\nu_{P}&\geq 2\sum_{j=1}^{q}\min \left\{\nu_{H_j(f)}, \nu_{H_j(g)}\right\}+(q-2)\sum_{j=1}^{q}\nu^{[1]}_{H_j(f)}\\
&\geq 2\sum_{j=1}^q\left (\nu^{[n]}_{H_j(f)}+\nu^{[n]}_{H_j(g)}-n\nu^{[1]}_{H_j(f)}+\ell'\nu_i\right)+(q-2)\sum_{j=1}^{q}\nu^{[1]}_{H_j(f)}\\
&= 2\sum_{j=1}^q(\nu^{[n]}_{H_j(f)}+\nu^{[n]}_{H_j(g)})+\frac{q-2n-2}{2}\sum_{j=1}^q(\nu^{[1]}_{H_j(f)}+\nu^{[1]}_{H_j(g)})+2\ell'\sum_{j=1}^q\nu_i\\
&\ge \frac{q+2n-2}{2n}\sum_{j=1}^q(\nu^{[n]}_{H_j(f)}+\nu^{[n]}_{H_j(g)})+2\ell'\sum_{j=1}^q\nu_i.
\end{split}
\end{align}
It is clear that $P\in B\big(q,0;f,g)$. By Lemma \ref{3.1}, we have
\begin{align*}
&\left(\frac{q+2n-2}{2n}(q-n-1)-q\right)\cdot\left(\frac{1}{\rho_f}+\frac{1}{\rho_g}\right)\le \frac{q+2n-2}{n}(\ell_f+\ell_g)\\
\Rightarrow\hspace{30pt} &\frac{q^2-(n+3)q-(n+1)(2n-2)}{q+2n-2}\le \frac{2\rho_f\rho_g(\ell_f+\ell_g)}{\rho_f+\rho_g}\\
\Rightarrow\hspace{30pt} &\frac{q(q-2)}{q+2n-2}\le n+1+\rho_{f,g}(\ell_f+\ell_g).
\end{align*}
This implies that
\begin{align}\label{4.4}
q\le \frac{q+2n-2}{q-2}\left(n+1+\rho_{f,g}(\ell_f+\ell_g)\right).
\end{align}

We now prove the theorem in the following two cases.

\noindent
\textbf{(a)} Assume that $\ell=1$. Since $q\ge 2n+2$, one has $\frac{q+2n-2}{q-2}\le 2$. Then, from (\ref{4.4}) we get 
$$q\le 2n+2+2\rho_{f,g}(\ell_f+\ell_g).$$
This is a contradiction. Therefore, the supposition is false. Hence, $f\equiv g$. 

\noindent
\textbf{(b)} Assume that $\ell>n$. Then $\ell'>0$. Because of Part a), it is enough for us to consider the case where $q\le 2n+2+2\rho_{f,g}(\ell_f+\ell_g).$

We set $h_i=\frac {H_i(f)}{H_i(g)}\ (1\le i\le q).$ Then
$\frac {h_i}{h_j} = \frac {H_i(f)\cdot H_j(g)}{H_j(f)\cdot H_i(g)}$
does not depend on the choice of the reduced representations of $f$ and $g$ respectively.

Take an arbitrary subset of $2n+2$ elements of the set $\{1,\ldots,q\}$, for instance it is $\{1,\ldots,2n+2\}$.
Denote by $\mathcal {I}$ the set of all combinations $I=(i_1,\ldots,i_{n+1})$ with $1\le i_1<...<i_{n+1}\le q$. For each $I=(i_1,\ldots,i_{n+1})\in \mathcal {I}$, put $h_I=\prod_{j=1}^{n+1}h_{i_j}$ and define 
$$A_I=(-1)^{\frac {(n+1)(n+2)}{2}+i_1+\cdots+i_{n+1}}\cdot \det (a_{i_rl};1\le r \le n+1,0\le l \le n)$$
$$\quad\quad\quad\quad\quad\quad\quad \quad\quad\quad\quad\quad\quad\quad\times\det (a_{j_sl};1\le s \le n+1,0\le l \le n),$$
where $J=(j_1,\ldots,j_{n+1})\in \mathcal {I}$ such that $I \cup J =\{1,2,\ldots,2n+2\}.$
Since 
$$\sum_{k=0}^n a_{ik}f_{k}-h_i\cdot \sum_{k=0}^n a_{ik}g_{k}=0\ (1\le i \le 2n+2),$$ 
one has
$$\det \ (a_{i0},\ldots,a_{in},a_{i0}h_i,\ldots,a_{in}h_i; 1\le i \le 2n+2)=0.$$
Therefore, $\sum_{I\in \mathcal {I}}A_Ih_I=0$ (note that $A_I\in\C^*$).

Take $I_0\in \mathcal {I}$. Denote by $t$ the minimal number satisfying the following: 
There exist $t$ elements $I_1,\ldots,I_t \in \mathcal {I}\setminus \{I_0\}$ and $t$ nonzero constants $b_i \in 
\C^*$ such that $\sum_{i=0}^tb_ih_{I_i}=0.$
By the minimality of $t$, the family $\{h_{I_1},\ldots,h_{I_t}\}$ is linearly independent over $\C$.

{\bf Case 1:} $t=1$. Then $\dfrac {h_{I_0}}{h_{I_1}}\in\C^*.$

{\bf Case 2:} $t\ge 2.$  Consider the linearly nondegenerate meromorphic mapping $F$ from $\B^m(R_0)$ into $\P^{t-1}(\C)$ with a reduced representation
$F=(h_{I_1}d:\cdots :h_{I_t}d),$ where $d$ is a meromorphic function. We see that
\begin{align*}
\sum_{i=0}^t\nu_{d h_{I_i}}^{[1]}(z)\le&\sum_{j=1}^{2n+2}\sharp\{i|j\in I_i,\nu_{H_j(f)}(z)>\nu_{H_j(g)}(z)\}\\
&\ \ \ +\sum_{i=1}^{2n+2}\sharp\{i|j\not\in I_i,\nu_{H_j(f)}(z)<\nu_{H_j(g)}(z)\}\\
=&\sum_{i=1}^{2n+2}\frac{p}{2}\nu_i(z)\le\frac{p}{2}\sum_{i=1}^q\nu_i(z),
\end{align*}
for every $z$ outside an analytic set of codimension two. Here by $\sharp S$ we denote the number of elements of the set $S$.

It is clear that $T_F(r,r_0)\le (n+1)(T_{f}(r,r_0)+T_{g}(r,r_0))$. Let $W(F)$ be a generalized Wronskian of $F$ and set 
$$G:=\prod_{0\le s<l\le 2}\left (\dfrac{(h_{I_l}d-h_{I_s}d)\cdot W(F)}{\prod_{i=0}^t(h_{I_i}d)}\right ).$$
Then we have $G\in B(0,3(t-1)(t+1)/2;F)\subset B(0,3(p-2)p/2;f,g)$. For each subset $J\subset\{1,\ldots,q\}$, set $J^c=\{1,\ldots,q\}\setminus J$. It is clear that 
$$\bigcup_{0\le s<l\le 2}((I_l\setminus I_s)\cup (I_s\setminus I_l))^c=\{1,\ldots,q\}.$$
We have
\begin{align*}
-\nu_G&=3\sum_{i=0}^t\nu_{h_{I_i}d}-3\nu_{W(F)}-\sum_{0\le s<l\le 2}\nu_{h_{I_l}d-h_{I_s}d}\\
&\le 3\sum_{i=0}^t\nu_{h_{I_i}d}^{[t-1]}- \sum_{0\le s<l\le 2}(\nu^0_{h_{I_l}/h_{I_s}-1})\\
&\le 3(t-1)\sum_{i=0}^t\nu_{h_{I_i}d}^{[1]}- \sum_{0\le s<l\le 2}\sum_{i\in ((I_l\setminus I_s)\cup (I_s\setminus I_l))^c}\nu^{[1]}_{H_i(f)}\\
&\le 3(p-2)\sum_{i=0}^t\nu_{h_{I_i}d}^{[1]}- \sum_{i=1}^q\nu^{[1]}_{H_i(f)}\le \frac{3(p-2)p}{2}\nu_i-\sum_{i=1}^q\nu^{[1]}_{H_i(f)}.
\end{align*}
Then, we have
\begin{align*}
\nu_{\prod_{i=1}^qP_i} &\ge\frac{q+2n-2}{2n}\sum_{i=1}^{q}(\nu_{H_i(f)}^{[n]}+\nu_{H_i(g)}^{[n]})+2\ell' \sum_{i=1}^{q}\nu_i\\
&\ge \frac{q+2n-2}{2n}\sum_{i=1}^{q}\left (\nu_{H_i(f)}^{[n]}+\nu_{H_i(g)}^{[n]}\right )+\frac{4\ell'}{3(p-2)p}\left (\sum_{i=1}^q\nu^{[1]}_{H_i(f)}-\nu_G\right).
\end{align*}
This yields that
$$ \nu_{G^{4\ell'} (\prod_{i=1}^qP_i)^{3(p-2)p}}\ge  \left(3(p-2)p\frac{q+2n-2}{2n}+\frac{2\ell'}{n}\right)\sum_{i=1}^{q}\left (\nu_{H_i(f)}^{[n]}+\nu_{H_i(g)}^{[n]}\right).$$
Note that $G\in B(0,3(p-2)p/2;f,g)$ and $P_i\in B(1,0;f,g)$. Then $G^{4\ell'}(\prod_{i=1}^qP_i)^{3(p-2)p}$ belongs to $B(3q(p-2)p,6\ell' (p-2)(p-1);f,g)$. From Lemma \ref{3.1}, we have
\begin{align*}
q&\le n+1+\frac{3q(p-2)p}{3(p-2)p\frac{q+2n-2}{2n}+\frac{2\ell'}{n}}+\rho_{f,g}\left(\ell_f+\ell_g+\frac{6\ell'(p-2)(p-1)}{3(p-2)p\frac{q+2n-2}{2n}+\frac{2\ell'}{n}}\right)\\
&=n+1+\frac{3(2n+2+2\rho_{f,g}(\ell_f+\ell_g))}{3\frac{q+2n-2}{2n}+\frac{2\ell'}{n(p-2)p}}+\rho_{f,g}\left(\ell_f+\ell_g+\frac{6\ell'(p-2)(p-1)}{6(p-2)p+\frac{2\ell'}{n}}\right)\\
&\le n+1+\frac{3(2n+2)}{6+\frac{2\ell'}{n(p-2)p}}+\rho_{f,g}\left(\ell_f+\ell_g+\frac{6\ell'(p-2)(p-1)+2(\ell_f+\ell_g)}{6(p-2)p+\frac{2\ell'}{n}}\right)\\
&=2n+1+\frac{3(p-2)p-\ell'}{3(p-2)p+\ell'/n}+\rho_{f,g}\left(\ell_f+\ell_g+\ell'\right).
\end{align*}
This is a contradiction. Hence, this case does not happen.

Therefore, for each $I\in \mathcal {I}$, there is $J\in \mathcal {I}\setminus \{I\}$ such that  $\dfrac {h_I}{h_J}\in\C^*.$

Consider the torsion free abelian subgroup generated by the family 
$\{[h_1] ,\ldots, [h_{q}]\}$ of the abelian group $\mathcal {M^*}_m/ \C^*.$ 
Then the family $\{[h_1],\ldots,[h_{q}]\}$ has the property $P_{q,n+1}$. 
By Proposition \ref{new2}, there exist $q-2n\ge 2$ elements, without loss of generality 
we may assume that they are  $[h_1],[h_i]$  such that $[h_1]=[h_i].$ Then
$\dfrac {h_1}{h_i}=\lambda\in\C^*.$
Suppose that $\lambda\not\equiv 1.$ Since $\dfrac{h_1(z)}{h_i(z)}=1$ for each $z\in\bigcup_{\overset{k=2}{k\ne i}}^{q}f^{-1}(H_k)\setminus(f^{-1}(H_1)\cup f^{-1}(H_i)),$ 
it implies that  $\bigcup_{\overset{k=2}{k\ne i}}^{q}f^{-1}(H_k)=\emptyset$. Hence $\sum_{\overset{k=2}{k\ne i}}^{q}\nu_{H_k(f)}^{[n]}=\sum_{\overset{k=2}{k\ne i}}^{q}\nu_{H_k(g)}^{[n]}=0$. Then, by Lemma \ref{3.1}, we have
\begin{align*} 
q-2\le n+1+\rho_{f,g}(\ell_f+\ell_g).
\end{align*}
This is a contradiction. Thus, $\lambda\equiv 1$, i.e., $h_1\equiv h_i$. Hence $\nu_{H_1(f)}=\nu_{H_1(g)}$ and
$\nu_{H_i(f)}=\nu_{H_i(g)}$. By the assumption, we note that $2<i<q$.

Now we consider 
\begin{align*}
P_{1}&=H_1(f)H_{2}(g)-H_{2}(f)H_1(g)=\frac{H_1(f)}{H_i(f)}\left (H_i(f)H_{2}(f)-H_{2}(f)H_i(g)\right )\not\equiv 0.
\end{align*}
From this inequality, we easily see that
\begin{align}\label{4.7}
\nu_{P_1}\ge (\nu_{H_1(f)}+\nu_{H_1(f)}^{[1]})+\nu_{H_2(f)}^{[n]}+\sum_{k=3}^{q}\nu_{H_k(f)}^{[1]}
\end{align}
and similarly
\begin{align}\label{4.8}
\begin{split}
\nu_{P_q}&\ge (\nu_{H_1(f)}+\nu_{H_1(f)}^{[1]})+\nu_{H_q(f)}^{[n]}+\sum_{k=2}^{q-1}\nu_{H_k(f)}^{[1]},\\
\nu_{P_{i-1}}&\ge (\nu_{H_i(f)}+\nu_{H_i(f)}^{[1]})+\nu_{H_{i-1}(f)}^{[n]}+\sum_{\underset{k\ne i-1,i}{k=1}}^q\nu_{H_k(f)}^{[1]},\\
\nu_{P_{i}}&\ge (\nu_{H_i(f)}+\nu_{H_i(f)}^{[1]})+\nu_{H_{i+1}(f)}^{[n]}+\sum_{\underset{k\ne i,i+1}{k=1}}^q\nu_{H_k(f)}^{[1]}.\\
\end{split}
\end{align}
Then, similar as (\ref{new}), with the help of (\ref{4.7}) and (\ref{4.8}), we have
\begin{align}\label{5.6}
\nu_P(z) \ge\frac{q+2n-2}{2n}\sum_{k=1}^{q}\left (\nu_{H_k(f)}^{[n]}(z)+\nu_{H_k(g)}^{[n]}(z)\right )+2\sum_{k=1,i}\nu_{H_k(f)}^{[1]}.
\end{align}

Consider the simple graph $\mathcal H'$, where the set of vertice is $\{1,\ldots,q\}\setminus\{1,i\}$ and the set of edges consists of all pairs $\{u,v\}$ such that $\frac{H_i(u)}{H_j(v)}\not\equiv\frac{H_u(g)}{H_v(g)}$. Since $f\not\equiv g$, the degree of $\mathcal H'$ at every vertex is at least $q-2-n\ge\frac{q-2}{2}$. By Dirac's theorem, $\mathcal H'$ has a Hamiltonian cycle $j_1,\ldots,j_{q-2},j_1$. Therefore, 
$$P_{u}'=H_{j_u}(f) H_{\sigma'(u)}(g)-H_{j_u}(g) H_{\sigma'(u)}(f)\not\equiv 0,$$
where $\sigma'(u)=j_{u+1}$ for $u<q-2$ and $\sigma'(q-2)=j_1$. We easily see that 
$$\nu_{P'_u}(z)\geq \sum_{k=u,\sigma'(u)}\min \left\{\nu_{H_{j_u}(f)}(z), \nu_{H_{j_u}(g)}(z)\right\}+\sum_{\underset{k\ne u,\sigma'(u)}{k=1}}^{q-2}\nu^{[1]}_{H_{j_u}(f)}(z)+\sum_{k=1,i}\nu^{[1]}_{H_k(f)}(z)$$
for all $z$ outside the analytic subset $\bigcup_{1\le u<v\le q}f^{-1}(H_u\cap H_v)$, which is of codimension two. Let $P'=\prod_{u=1}^{q-2}P_u'$ and similar as (\ref{new}), we get
\begin{align*}
\nu_{P'}(z) &\ge\frac{q+2n-4}{2n}\sum_{\overset{k=2}{k\ne i}}^q\left (\nu_{H_k(f)}^{[n]}(z)+\nu_{H_k(g)}^{[n]}(z)\right )+(q-2)\sum_{k=1,i}\nu_{H_k(f)}^{[1]}(z)\\
&=\frac{q+2n-4}{2n}\sum_{k=1}^q\left (\nu_{H_k(f)}^{[n]}(z)+\nu_{H_k(g)}^{[n]}(z)\right )-(2n-2)\sum_{k=1,i}\nu_{H_k(f)}^{[1]}(z).
\end{align*}
It is clear that $P^{n-1}P'\in B(nq-2,0;f,g)$ and satisfying
$$ \nu_{P^{n-1}P'}\ge\frac{n(q+2n-2)-2}{2n}\sum_{k=1}^{q}\left (\nu_{H_k(f)}^{[n]}(z)+\nu_{H_k(g)}^{[n]}(z)\right ).$$
Then from Lemma \ref{3.1}, we have
\begin{align*}
q&\le n+1+\rho_{f,g}(\ell_f+\ell_g)+ \frac{2n(nq-2)}{n(q+2n-2)-2}\\
&\le n+1+\rho_{f,g}(\ell_f+\ell_g)+\frac{2n(2n^2+2n-2)}{4n^2-2}+\frac{4n^2\rho_{f,g}(\ell_f+\ell_g)}{4n^2-2}\\
&=2n+1+\frac{2n-1}{2n-1/n}+\rho_{f,g}\frac{4n^2-1}{2n^2-1}(\ell_f+\ell_g).
\end{align*}
This is a contradiction. 

Hence, we must have $f\equiv g$. The theorem is proved. 
\end{proof}

\section{algebraic dependence problem}
\begin{lemma}[{See  \cite[Lemma 3.1]{Q22}}]\label{5.1}
Let $f^1,f^2,\ldots,f^k$ be as in Theorem \ref{1.3} and $M=\B^m(R_0)$. Assume that each $f^u$ has a reduced representation $f^u=(f^u_{0}:\cdots :f^u_{n})$, $1\le u\le k$. Suppose that there exist integers $1\le i_1<i_2<\cdots <i_k\le q$ such that
$$ P:=\mathrm{det}\left(H_{i_j}(f^u)\right)_{1\le j,u\le k}\not\equiv 0.$$
Then we have
\begin{align*}
\nu_P (z)\ge\sum_{j=1}^k\left(\min_{1\le u\le k}\{\nu_{H_{i_j}(f^u)}(z)\}-\nu^{[1]}_{H_{i_j}(f^1)}(z)\right)+ (k-1)\sum_{i=1}^q\nu^{[1]}_{H_{i}(f^1)}(z),
\end{align*}
for every $z\in\B^m(R_0)$ outside an analytic set of codimension two.
\end{lemma}

\begin{proof}[Proof of Theorem \ref{1.3}]
As usual, we may suppose that $M=\B^m(\R_0)$. For each $1\le i\le q,$ we put $V_i:=( (f^1,H_i),(f^2,H_i),(f^3,H_i))$.  We write $V_i\cong V_j$ if $V_i\wedge V_j\equiv 0$, otherwise we write $V_i\not\cong V_j.$ 

Suppose that $f^1\wedge f^2\wedge f^3\not\equiv 0$. Without loss of generality, we may assume that
$$\underbrace{V_1\cong \cdots\cong V_{l_1}}_{\text { group } 1}\not\cong
\underbrace{V_{l_1+1}\cong\cdots\cong V_{l_2}}_{\text { group } 2}\not\equiv \underbrace{V_{l_2+1}\cong \cdots\cong V_{l_3}}_{\text { group } 3}\not\cong\cdots \not\cong\underbrace{V_{l_{s-1}}\cong\cdots\cong V_{l_{s}}}_{\text { group } s},$$
where $l_s=q.$ 
 For each $1\le i\le q$, we set
$$ \sigma (i)=\begin{cases}
i+n&\text{ if }i+n\le q,\\
i+n-q&\text{ if }i+n>q.
\end{cases} $$
Since each group has at most $n$ elements, $V_i$ and $V_{\sigma (i)}$ belong to two distinct groups, i.e., $V_i\wedge V_{\sigma (i)}\not\equiv 0$. Then, we may choose another index, denoted by $\gamma (i)$, such that 
$$V_i\wedge V_{\sigma (i)}\wedge V_{\gamma (i)}\not\equiv 0.$$
We set
$$ P_i:=\det\left(H_i(f^u),H_{\sigma (i)}(f^u),H_{\gamma (i)}(f^u);1\le u\le 3\right)\not\equiv 0.$$
Then, by Lemma \ref{5.1} we have
\begin{align*}
\nu_{P_i}&\ge  \sum_{j=i,\sigma(i)}\left(\min_{1\le u\le 3}\nu_{H_j(f^u)}-\nu^{[1]}_{H_j(f^1)}\right)+2\sum_{j=1}^q\nu^{[1]}_{H_j(f^1)}\\
&\ge \sum_{j=i,\sigma(i)}\left(\sum_{u=1}^3\nu^{[n]}_{H_j(f^u)}-(2n+1)\nu^{[1]}_{H_j(f^1)}\right)+2\sum_{j=1}^q\nu^{[1]}_{H_j(f^1)}.
\end{align*}
Summing-up both sides of the above inequality over all $1\le i\le q$, we have
$$\nu_{\prod_{i=1}^qP_i}\ge 2\sum_{u=1}^3\sum_{j=1}^q\nu^{[n]}_{H_j(f^u)}+(2q-4n-2)\sum_{j=1}^q\nu^{[1]}_{H_j(f^u)}\ge\frac{2q+2n-2}{3n}\sum_{u=1}^3\sum_{j=1}^q\nu^{[n]}_{H_j(f^u)}.$$
It is easy to see that $\prod_{i=1}^qP_i\in B(q,0;f^1,f^2,f^3).$
Then, by Lemma \ref{3.1} we have
$$ \frac{2q+2n-2}{3n}(q-n-1)-q\le 2\frac{2q+2n-2}{3n}\left(\sum_{u=1}^3\frac{1}{\rho_{f^u}}\right)^{-1}\sum_{u=1}^3\ell_{f^u},$$
i.e.,
$$ q\le n+1+\frac{3nq}{2q+2n-2}+2\left(\sum_{u=1}^3\frac{1}{\rho_{f^u}}\right)^{-1}\sum_{u=1}^3\ell_{f^u}.$$
This is a contradiction. Hence $f^1\wedge f^2\wedge f^3\equiv 0$. The theorem is proved.
\end{proof}

\begin{proof}[Proof of Theorem \ref{1.4}]

Denote by $\mathcal I$ the set of all $k$-tuples $I=(i_1,\ldots,i_k)\in \N^k$ with $1\le i_1<i_2<\cdots <i_k\le q$. Suppose contrarily that $f^1\times f^2\times\cdots\times f^k$ is not algebraically degenerate. Then for every $I=(i_1,\ldots,i_k)\in\mathcal I$,
$$ P_{I}:=\det \left (H_{i_j}(f^u)\right)_{1\le j,u\le k}\not\equiv 0. $$
By Lemma \ref{5.1}, we have
\begin{align*}
\nu_{P_I}&\ge\sum_{j=1}^k\left (\min_{1\le u\le k}\nu_{H_{i_j}(f^u)}-\nu^{[1]}_{H_{i_j}(f^1)}\right)+(k-1)\sum_{i=1}^q\nu^{[1]}_{H_i(f)}\\
&\ge\sum_{j=1}^k\left (\nu^{[n]}_{H_{i_j}(f^1)}-\nu^{[1]}_{H_{i_j}(f^1)}\right)+(k-1)\sum_{i=1}^q\nu^{[1]}_{H_i(f)}.
\end{align*}
Setting $P=\prod_{I\in\mathcal I}P_I$ and summing up both sides of the above inequality over all $I\in\mathcal I$, we get
\begin{align}\label{3.14}
\begin{split}
\nu_P&\ge\sharp\mathcal I\cdot\sum_{i=1}^q\left (\frac{k}{q}\nu^{[n]}_{H(f^1)}+\frac{((k-1)q-k)}{q}\nu^{[1]}_{H_{i_j}(f^1)}\right)\\
&=\sharp\mathcal I\cdot\left (\frac{1}{q}+\frac{((k-1)q-k)}{knq}\right)\sum_{u=1}^k\sum_{i=1}^q\nu^{[n]}_{H_i(f^u)}.
\end{split}
\end{align}
Applying Lemma \ref{3.1} for the function $P\in B(\sharp\mathcal I,0;f^1,\ldots,f^k)$, we get
$$\sum_{u=1}^k\frac{\sharp\mathcal I\cdot\left (\frac{1}{q}+\frac{((k-1)q-k)}{knq}\right)(q-n-1)-\sharp\mathcal I}{\rho_{f^u}}-2\sum_{u=1}^k\sharp\mathcal I\cdot \left (\frac{1}{q}+\frac{((k-1)q-k)}{knq}\right) \ell_{f^u}\le 0, $$
i.e.,
$$ q\le n+1+\frac{knq}{(k-1)q+k(n-1)}+2\left(\sum_{u=1}^k\frac{1}{\rho_{f^u}}\right)^{-1}\sum_{u=1}^k\ell_{f^u}.$$
This is a contradition. Therefore, $f^1\times\cdots\times f^k$ is algebraically degenerate. The theorem is proved.
\end{proof}

\noindent{\bf Acknowledgements.} This work was done during a stay of the author at the Vietnam Institute for Advanced Study in Mathematics (VIASM). He would like to thank the staff there, in particular the partially support of VIASM.  This research is funded by Vietnam National Foundation for Science and Technology Development (NAFOSTED) under grant number 101.02-2021.12.

\noindent
{\bf Disclosure statement.} No potential conflict of interest was reported by the author(s).

\vskip0.1cm
{\footnotesize 
\noindent
{\sc Si Duc Quang}\\
$^1$Department of Mathematics, Hanoi National University of Education,\\
136-Xuan Thuy, Cau Giay, Hanoi, Vienam.\\
$^2$Thang Long Institute of Mathematics and Applied Sciences\\
Nghiem Xuan Yem, Hoang Mai, Hanoi, Vietnam.\\
\textit{E-mail}: quangsd@hnue.edu.vn}
\end{document}